\documentclass{article}%
\usepackage{amssymb}
\usepackage{amsmath}
\usepackage{amsfonts}
\usepackage{graphicx}%
\providecommand{\U}[1]{\protect\rule{.1in}{.1in}}
\newtheorem{theorem}{Theorem}

\newtheorem{lemma}[theorem]{Lemma}

\newtheorem{proposition}[theorem]{Proposition}
\newtheorem{remark}[theorem]{Remark}

\newenvironment{proof}[1][Sketch of the proof]{\noindent\textbf{#1.} }{\ \rule{0.5em}{0.5em}}
\begin{document}

\title{Stochastic and statistical stability of the classical Lorenz flow under
perturbations modeling anthropogenic type forcing}
\author{Michele Gianfelice$%
{{}^\circ}%
$\\{\small $%
{{}^\circ}%
$Dipartimento di Matematica e Informatica, Universit\`{a} della Calabria} \\{\small Ponte Pietro Bucci, cubo 30B, I-87036 Arcavacata di Rende (CS)} \\{\small \texttt{gianfelice@mat.unical.it}} }
\maketitle

\begin{abstract}
We review the results obtained in \cite{GMPV} and \cite{GV} on the stochastic
and statistical stability of the classical Lorenz flow, where, looking at the
Lorenz'63 ODE system as a simple - yet non trivial - model of the atmospheric
circulation, the perturbation schemes introduced in these papers are designed
to represent the effect of the so called anthropogenic forcing on the dynamics
of the atmosphere.

\end{abstract}
\tableofcontents

\bigskip

\begin{description}
\item[AMS\ subject classification:] {\small 34F05, 93E15. }

\item[Keywords and phrases:] {\small Random perturbations of dynamical
systems, classical Lorenz flow, random dynamical systems, semi-Markov random
evolutions, piecewise deterministic Markov processes, Lorenz'63 model,
anthropogenic forcing.}

\item[Ethics declaration:] {\small the corresponding author states that there is no conflict of interest.}
\end{description}

\bigskip

\section{Introduction}

\subsection{The classical Lorenz flow}

Turbulent systems such the atmosphere are usually modeled by flows exhibiting
a sensitive dependence on the initial conditions. The behaviour of the
trajectories of the system in the phase space for large times is usually
numerically very hard to compute and consequently the same computational
difficulty affects also the computation of the phase averages of physically
relevant observables. A way to overcome this problem is to select a few of
these relevant observables under the hypothesis that the statistical
properties of the smaller system defined by the evolution of such quantities
can capture the important features of the statistical behaviour of the
original system \cite{NVKDF}.

This turns out to be the case when considering \emph{classical Lorenz model},
or, in the physics literature, \emph{Lorenz'63 model}, that is the system of
equation
\begin{equation}
\left\{
\begin{array}
[c]{l}%
\dot{x}_{1}=-\zeta x_{1}+\zeta x_{2}\\
\dot{x}_{2}=-x_{1}x_{3}+\gamma x_{1}-x_{2}\\
\dot{x}_{3}=x_{1}x_{2}-\beta x_{3}%
\end{array}
\right.  \ , \label{L}%
\end{equation}
which was introduced by E. Lorenz in his celebrated paper \cite{Lo} as a
simplified yet non trivial model for thermal convection of the atmosphere and,
since then, it has been pointed out as the typical real example of a
non-hyperbolic three-dimensional flow whose trajectories show a sensitive
dependence on initial conditions. More precisely, the classical Lorenz flow,
for $\zeta=10,\gamma=28,\beta=8/3,$ has been proved in \cite{Tu}, and more
recently in \cite{AM}, to show the same dynamical features of its ideal
counterpart the so called \emph{geometric Lorenz flow}, introduced in
\cite{ABS} and in \cite{GW}, which represents the prototype of a
three-dimensional flow exhibiting a partially hyperbolic attractor \cite{AP}.

\subsection{Physical motivations for the study of the stability of the
statistical properties of the classical Lorenz flow}

The analysis of the stability of the statistical properties of the classical
Lorenz flow can provide a theoretical framework for the study of climate
changes, in particular those induced by the anthropogenic influence on climate dynamics.

A possible way to study this problem is to add a weak perturbing term to the
phase vector field generating the atmospheric flow which model the atmospheric
circulation: the so called \emph{anthropogenic forcing}. Assuming that the
atmospheric circulation is described by a model exhibiting a robust singular
hyperbolic attractor, as it is the case for the classical Lorenz flow, it has
been shown empirically that the effect of the perturbation can possibly affect
just the statistical properties of the system \cite{Pa}, \cite{CMP}.
Therefore, because of its very weak nature (small intensity and slow
variability in time), a practical way to measure the impact of the
anthropogenic forcing on climate statistics is to look at the extreme value
statistics of those particular observables whose evolution may be more
sensitive to it \cite{Su}. In the particular case these observables are given
by bounded (real valued) functions on the phase space, an effective way to
look at their extreme value statistics is to look first at the statistics of
their extrema and then eventually to the extreme value statistics of these
making use, for example, of the techniques described in \cite{Letal}.

\subsubsection{Stability of the invariant measure of the classical Lorenz
flow}

Since $C^{1}$ perturbations of the classical Lorenz vector field admit a
$C^{1+\epsilon}$ stable foliation \cite{AM} and since the geometric Lorenz
attractor is robust in the $C^{1}$ topology \cite{AP}, it is natural to
discuss the statistical and the stochastic stability of the classical Lorenz
flow under this kind of perturbations.

As a matter of fact, in applications to climate dynamics, when considering the
Lorenz'63 flow as a model for the atmospheric circulation, the analysis of the
stability of the statistical properties of the unperturbed flow under
perturbations of the velocity phase field of this kind can turn out\ to be a
useful tool in the study of the so called \emph{anthropogenic climate change}
\cite{CMP}.

\section{Statistical stability\label{statstab}}

Since the SRB measure of the geometric Lorenz flow can be constructed starting
from the invariant measure of the one-dimensional map obtained through
reduction to the quotient leaf space of the Poincar\'{e} map on a
two-dimensional manifold transverse to the flow \cite{AP}, the statistical
stability for the invariant measure of this map implies that of the SRB
measure of the unperturbed flow. Results in this direction are given in
\cite{AS}, \cite{BR} and \cite{GL} where strong statistical stability of the
geometric Lorenz flow is analysed.

For what concerns the classical Lorenz flow in \cite{GMPV} it has been shown
that the effect of an additive constant perturbation term to the classical
Lorenz vector field results into a particular kind of perturbation of the map
of the interval describing the evolution of the maxima of the Casimir function
for the (+) Lie-Poisson brackets associated to the $so\left(  3\right)  $
algebra. Moreover, it has been proved that the invariant measures for the
perturbed and for the unperturbed 1-$d$ maps of this kind have Lipschitz
continuous density and that the unperturbed invariant measure is strongly
statistically stable.

More precisely, the vector field (\ref{L}) has the interesting feature that it
can be rewritten as
\begin{equation}
\left\{
\begin{array}
[c]{l}%
\dot{y}_{1}=-\zeta y_{1}+\zeta y_{2}\\
\dot{y}_{2}=-y_{1}y_{3}-\gamma y_{1}-y_{2}\\
\dot{y}_{3}=y_{1}y_{2}-\beta y_{3}-\beta\left(  \gamma+\zeta\right)
\end{array}
\right.  \ , \label{L1}%
\end{equation}
showing the corresponding flow to be generated by the sum of a Hamiltonian
$SO\left(  3\right)  $-invariant field and a gradient field (we refer the
reader to \cite{GMPV} and references therein). Therefore, as it has been
proved in \cite{GMPV}, the invariant measure of the classical Lorenz flow can
be constructed starting from the invariant measure of a map of the interval
$T,$ whose graph (see fig. 1) looks like that originally computed by Lorenz
(\cite{Lo} fig. 4) describing the evolution of the extrema of the first
integrals of the associated Hamiltonian flow \cite{PM} such as, for example,
the just mentioned Casimir function $C\left(  t\right)  :=\left\vert
Y^{t}\right\vert ^{2},$ where $\left\vert \cdot\right\vert $ is the Euclidean
norm in $\mathbb{R}^{3}$ and $\left(  Y^{t}\ ,t\geq0\right)  $ denotes the
flow defined by (\ref{L1}). In this case, denoting by $\phi_{0}$ the vector
field defined in (\ref{L1}), since $\dot{C}=\phi_{0}\cdot\nabla C$ and
$\ddot{C}=\phi_{0}\cdot\nabla\left(  \phi_{0}\cdot\nabla C\right)  ,$ setting,
for $\epsilon$ sufficiently small,%
\begin{equation}
\mathcal{M}_{\epsilon}:=\left\{  \left(  y_{1},y_{2},y_{3}\right)
\in\mathbb{R}^{3}:\left\vert y_{1}\right\vert \leq\epsilon,\left\vert
y_{2}\right\vert \leq\epsilon,y_{3}\in\left[  -\left(  \gamma+\zeta\right)
,\epsilon-\left(  \gamma+\zeta\right)  \right]  \right\}
\end{equation}
the Poincar\'{e} surface taken into account is
\begin{equation}
\mathcal{M}:=\left\{  \left(  y_{1},y_{2},y_{3}\right)  \in\mathcal{M}%
_{\epsilon}:\left(  \phi_{0}\cdot\nabla C\right)  \left(  y_{1},y_{2}%
,y_{3}\right)  =0,\phi_{0}\cdot\nabla\left(  \phi_{0}\cdot\nabla C\right)
\left(  y_{1},y_{2},y_{3}\right)  \leq0\right\}
\end{equation}
Note that $\mathcal{M}$ contains the hyperbolic critical point $c_{0}:=\left(
0,0,-\left(  \gamma+\zeta\right)  \right)  $ of $\phi_{0}.$ Denoting by
$\bar{T}$ the automorphism of the set of the leaf of the invariant foliation
defined by the intersection of the stable manifolds of the flow with
$\mathcal{M},T$ is then obtained identifying the leaves of the invariant
foliation corresponding to the same values of $C.$

\subsection{The Lorenz-like cusp map $T$ and its invariant measure}

The local behaviors of the branches of $T$ are ($c$ will denote a positive
constant which could take different values from one formula to another):
\begin{align}
&  \left\{
\begin{array}
[c]{l}%
T(x)=\alpha^{\prime}x+\beta^{\prime}x^{1+\psi}+o(x^{1+\psi});\ x\rightarrow
0^{+}\\
DT(x)=\alpha^{\prime}+cx^{\psi}+o(x^{\psi}),\ \alpha^{\prime}>1;\ \beta
^{\prime}>0;\ \psi>1
\end{array}
\right.  \ ,\label{T_DT1}\\
&  \left\{
\begin{array}
[c]{l}%
T(x)=\alpha(1-x)+\tilde{\beta}(1-x)^{1+\kappa}+o((1-x)^{1+\kappa
});\ x\rightarrow1^{-}\\
DT(x)=-\alpha-c(1-x)^{\kappa}+o((1-x)^{\kappa}),\ 0<\alpha<1,\ \tilde{\beta
}>0,\ \kappa>1
\end{array}
\right.  \ ,\label{T_DT2}\\
&  \left\{
\begin{array}
[c]{l}%
T(x)=1-A^{\prime}(x_{0}-x)^{B^{\prime}}+o((x_{0}-x)^{B^{\prime}}%
);\ x\rightarrow x_{0}^{-},\ A^{\prime}>0\\
DT(x)=c(x_{0}-x)^{B^{\prime}-1}+o((x_{0}-x)^{B^{\prime}-1}),\ 0<B^{\prime}<1
\end{array}
\right.  \ ,\label{T_DT3}\\
&  \left\{
\begin{array}
[c]{l}%
T(x)=1-A(x-x_{0})^{B}+o((x-x_{0})^{B});\ x\rightarrow x_{0}^{+},\ A>0\\
DT(x)=-c(x-x_{0})^{B-1}+o((x-x_{0})^{B-1}),\ 0<B<1
\end{array}
\right.  \ . \label{T_DT4}%
\end{align}
From this we deduce that $T^{-1}$ is $C^{1+\iota},$ for some $\iota\in\left(
0,1\right)  ,$ on each open interval $(0,x_{0}),(x_{0},1).$ Indeed, by the
result in \cite{AM}, the stable foliation for the classical Lorenz flow is
$C^{1+\alpha}$ for some $\alpha\in\left(  0.278,1\right)  ,$ which means, by
(\ref{T_DT3}) and (\ref{T_DT4}), that, for any $x\in(0,x_{0}),T^{\prime
}\left(  x\right)  =\left\vert x_{0}-x\right\vert ^{1-B^{\prime}}\left[
1+G_{1}\left(  x\right)  \right]  $ with $G_{1}\in C^{\alpha B^{\prime}%
}(0,x_{0})$ and, for any $x\in(x_{0},1),T^{\prime}\left(  x\right)
=\left\vert x-x_{0}\right\vert ^{1-B}\left[  1+G_{2}\left(  x\right)  \right]
$ with $G_{2}\in C^{\alpha B}(x_{0},1).$ In particular this implies that for
any couple of points $x,y$ belonging either to $(0,x_{0})$ or to $(x_{0},1)$
\begin{equation}
|T^{\prime}(x)-T^{\prime}(y)|\leq C_{h}\left\vert T^{\prime}(x)\right\vert
\left\vert T^{\prime}(y)\right\vert \left\vert x-y\right\vert ^{\iota}\ ,
\label{H}%
\end{equation}
where $\iota\in(0,1-B^{\ast}],$ with $B^{\ast}:=B\vee B^{\prime},$ and the
constant $C_{h}$ is independent of the location of $x$ and $y.$

\begin{figure}[ptbh]
\centering
\resizebox{0.75\textwidth}{!}{\includegraphics{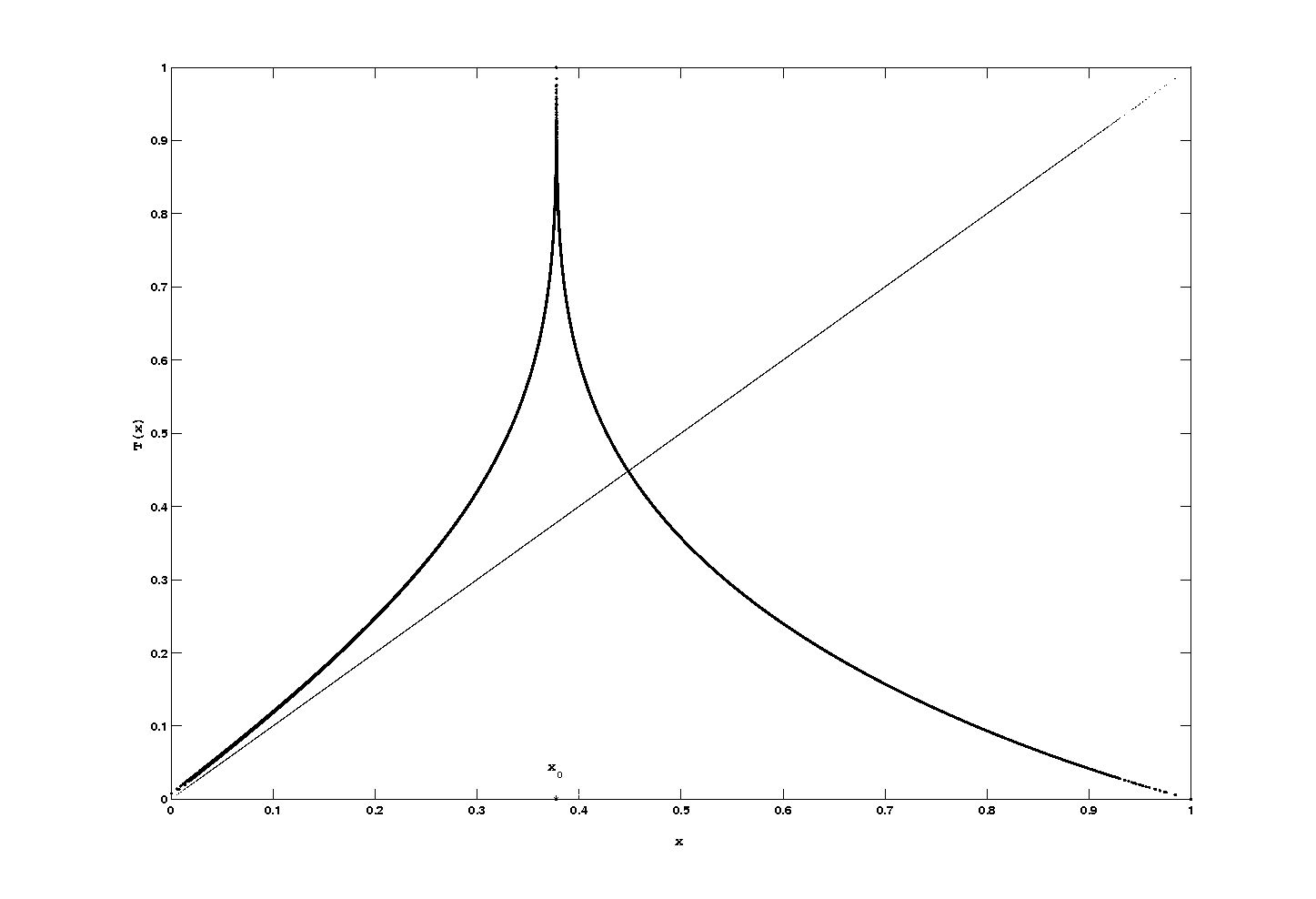}
}
\caption{Normalized Lorenz cusp map for the Casimir maxima}%
\label{fig:1}%
\end{figure}

\begin{proposition}
The density $\rho$ of the invariant measure $\mu$ is Lipschitz continuous and
bounded over the intervals $[0,1]$. Moreover,%
\begin{equation}
\lim_{x\rightarrow0^{+}}\rho(x)=\lim_{x\rightarrow1^{-}}\rho(x)=0\ .
\end{equation}

\end{proposition}

\begin{proof}
Let us set: $a_{0}:=T_{2}^{-1}x_{0};$\ $a_{0}^{\prime}:=T_{1}^{-1}x_{0}%
;$\ $a_{p}^{\prime}=T_{1}^{-p}a_{0}^{\prime};$\ $a_{p}=T_{2}^{-1}%
T_{1}^{-(p-1)}a_{0}^{\prime},$\ $p\geq1.$ We also define the sequences
$\{b_{p}\}_{p\geq1}\subset(x_{0},a_{0})$ and $\{b_{p}^{\prime}\}_{p\geq
1}\subset(a_{0}^{\prime},x_{0})$ as $Tb_{p}^{\prime}=Tb_{p}=a_{p-1}.$ Hence,
we can induce on $I:=(a_{0}^{\prime},a_{0})\backslash\{x_{0}\}$ and to replace
the action of $T$ on $I$ with that of the first return map $T_{I}$ into $I$
and prove that the systems $(I,T_{I})$ will admit an absolutely continuous
invariant measure $\mu_{I}$ which is in particular equivalent to the Lebesgue
measure with a density $\rho_{I}$ bounded from below and from above. To do
this, following \cite{CHMV}, we also need to induce over the open sets
$(a_{n}^{\prime},a_{n-1}^{\prime})$ and $(a_{n},a_{n+1}),n>1,$ simply denoted
in the following as the rectangles $I_{n},$ provided we show that the induced
maps are aperiodic uniformly expanding Markov maps with bounded distortion on
each set with prescribed return time. On the sets $I_{n}$ the first return map
$T_{I_{n}}$ is Bernoulli, while the aperiodicity condition on $I$ follows by
direct inspection of the graph of the first return map $T_{I}%
:I\circlearrowleft$ showing that it maps: $(a_{0}^{\prime},b_{1}^{\prime})$
onto $(x_{0},a_{0});$ the intervals $(b_{l}^{\prime},b_{l+1}^{\prime}%
),\ l\geq1,$ onto the interval $(a_{0}^{\prime},x_{0});$ the interval
$(b_{1},a_{0})$ onto $(x_{0},a_{0})$ and finally the intervals $(b_{l+1}%
,b_{l}),\ l\geq1$ onto $(a_{0}^{\prime},x_{0}).$ The proof of the boundedness
of the distortion is analogous to that given in Proposition 3 of \cite{CHMV}
and rely on the proof that the first return maps are uniformly expanding. In
particular, in the initial formula (5) in \cite{CHMV} we need now to replace
the term $\left\vert \frac{D^{2}T(\xi)}{DT(\xi)}\right\vert |T^{q}\left(
x\right)  -T^{q}\left(  y\right)  |,$ where $\xi$ is a point between
$T^{q}\left(  x\right)  $ and $T^{q}\left(  y\right)  ,$ with $\frac
{1}{|DT(\xi)|}C_{h}|DT(T^{q}\left(  x\right)  )||DT(T^{q}\left(  y\right)
)||T^{q}\left(  x\right)  -T^{q}\left(  y\right)  |^{\iota}$ which is smaller
than $C_{h}\left(  |DT(T^{q}\left(  x\right)  )|\vee|DT(T^{q}\left(  y\right)
)|\right)  |T^{q}\left(  x\right)  -T^{q}\left(  y\right)  |^{\iota}$ by
monotonicity of $\left\vert DT\right\vert .$ The key estimate (11) in
\cite{CHMV} will reduce in our case to the bound of the quantity $\sup_{\xi
\in\lbrack b_{i+1},b_{i}]}|DT\left(  \xi\right)  ||b_{i}-b_{i+1}|.$ By using
for $DT$ the expressions given in the formulas (\ref{T_DT3}) and
(\ref{T_DT4}), and for the $b_{i}$ the scaling $(x_{0}-b_{p}^{\prime}%
)\sim\frac{c}{\left(  \alpha^{\prime}\right)  ^{\frac{p}{B^{\prime}}}%
};\ (b_{p}-x_{0})\sim\frac{c}{\left(  \alpha^{\prime}\right)  ^{\frac{p}{B}}}$
(see formula (75) of \cite{GMPV}) we immediately get that the above quantity
is of order $\frac{1}{(\alpha^{\prime})^{i}},$ which is enough to pursue the
argument about the estimate of the distortion presented in \cite{CHMV}.

The invariant measure $\mu_{I}$ for the induced map $T_{I}$ is related to the
invariant measure $\mu$ over the whole interval by the Pianigiani formula
\begin{equation}
\mu(B)=C_{r}\sum_{i}\sum_{j=0}^{\tau_{i}-1}\mu_{I}(T^{-j}(B)\cap Z_{i})
\label{reldens}%
\end{equation}
where $B$ is any Borel set in $[0,1]$ and the first sum runs over the
cylinders $Z_{i}$
\begin{align}
Z_{1}  &  =(a_{0}^{\prime},b_{1}^{\prime})\cup(b_{1},a_{0})\label{SI}\\
Z_{i}  &  =(b_{i-1}^{\prime},b_{i}^{\prime})\cup(b_{i},b_{i-1})\quad i\geq2\ ,
\label{SII}%
\end{align}
with prescribed first return time $\tau_{i}$ and whose union gives $I.$ The
normalizing constant $C_{r}=\mu(I)$ satisfies $1=C_{r}\sum_{i}\tau_{i}\mu
_{I}(Z_{i}).$ This immediately implies that by calling $\hat{\rho}$ the
density of $\mu_{I}$ we have that $\rho(x)=C_{r}\hat{\rho}(x)$ for Lebesgue
almost every $x\in I$ and therefore $\rho$ can be extended to a Lipschitz
continuous function on $I$ as $\hat{\rho}.$
\end{proof}

Anyway, we remark that the existence of an invariant measure for $T$ follows
also by combining Theorem 2 in \cite{Pi2} and the results in section 4.2 of
\cite{Bu} since one can check by direct computation that the map $T:=W\circ
T\circ W^{-1},$ where $W$ is the distribution function associated to the
probability measure on $\left(  \left[  0,1\right]  ,\mathcal{B}\left(
\left[  0,1\right]  \right)  \right)  $ with density
\begin{equation}
\left[  0,1\right]  \ni x\longmapsto W^{\prime}\left(  x\right)
:=N_{\bar{\gamma},\bar{\beta}}e^{-\bar{\gamma}x}x^{\bar{\beta}}\left(
1-x\right)  ^{\bar{\beta}}%
\end{equation}
(see formulas (83) and (84) in \cite{GMPV}) for suitably chosen parameters
$\bar{\gamma},\bar{\beta}>0,$ is such that $\inf\left\vert \overline
{T}^{\prime}\right\vert >1.$

\subsubsection{Statistical stability of $T$}

Let us denote by $\lambda$ the Lebesgue measure on $\left(  \mathbb{R}%
,\mathcal{B}\left(  \mathbb{R}\right)  \right)  $ and by $T_{\epsilon}$ the
perturbed map. We show that under the following assumptions the density
$\rho_{\epsilon}$ of the perturbed measure will converge to the density $\rho$
of the unperturbed one in the $L_{\lambda}^{1}$ norm.

\begin{itemize}
\item[\emph{Assumption A}] $T_{\epsilon}$ is a Markov map of the unit interval
which is one-to-one and onto on the intervals $[0,x_{\epsilon,0})$ and
$(x_{\epsilon,0},1],$ convex on both sides and of class $C^{1+\iota_{\epsilon
}}$ on the open interval $(0,x_{\epsilon,0})\cup(x_{\epsilon,0},1).$

\item[\emph{Assumption B}] Let $\left\Vert \cdot\right\Vert _{0}$ denotes the
$C^{0}$-norm on the unit interval, then
\begin{equation}
\lim_{\epsilon\rightarrow0}\left\Vert T_{\epsilon}-T\right\Vert _{0}=0\ .
\end{equation}
Moreover, $\forall x\in\lbrack0,1],\ x\neq x_{0},$ we can find $\epsilon(x)$
such that,\linebreak$\forall\epsilon<\epsilon(x),\ DT_{\epsilon}$ exists and
is finite and we have
\begin{equation}
\lim_{\epsilon\rightarrow0}DT_{\epsilon}(x)=DT(x)\ .
\end{equation}
Furthermore,
\begin{equation}
\lim_{x\rightarrow x_{0}^{+}}\lim_{\epsilon\rightarrow0}\frac{DT_{\epsilon
}(x)}{DT(x)}=\lim_{x\rightarrow x_{0}^{-}}\lim_{\epsilon\rightarrow0}%
\frac{DT_{\epsilon}(x)}{DT(x)}=1\ .
\end{equation}

\item[\emph{Assumption C}] Let us denote by $C_{h,\epsilon}$ and
$\iota_{\epsilon}$ respectively the H\"{o}lder constant and the H\"{o}lder
exponent for the derivative of $T_{\epsilon}$ on the open interval
$(0,x_{\epsilon,0})\cap(x_{\epsilon,0},1);$ namely: $|DT_{\epsilon
}(x)-DT_{\epsilon}(y)|\leq C_{h,\epsilon}|x-y|^{\iota_{\epsilon}}$ for any
$x,y$ either in $(0,x_{\epsilon,0})$ or in $(x_{\epsilon,0},1).$ We assume
$C_{h,\epsilon}$ and $\iota_{\epsilon}$ to converge to the corresponding
quantities for $T$ in the limit $\epsilon\rightarrow0.$

\item[\emph{Assumption D}] Let us set $d_{(\epsilon,1,0)}:=\inf_{(b_{\epsilon
,1},a_{\epsilon,0})}|DT_{\epsilon}(x)|.$ We assume $d_{(\epsilon,1,0)}>1$ and
that there exists a constant $d_{c}$ and $\epsilon_{c}=\epsilon\left(
d_{c}\right)  $ such that,\newline$\forall\epsilon<\epsilon_{c},\ |d_{(1,0)}%
-d_{(\epsilon,1,0)}|<d_{c}.$
\end{itemize}

Examples of $T_{\epsilon}$ are plotted in fig. 2 (see also fig. 3 in
\cite{GMPV}). We remark that these realizations of the perturbed map
$T_{\epsilon}$ occur when a constant perturbing vector field of order
$\epsilon$ is added to (\ref{L1}), just as it has been empirically shown in
the physics literature (see e.g. \cite{CMP}) it turns out to be the case when
considering the effect on the atmospheric circulation of greenhouse gases.
Clearly under these assumptions the map $T_{\epsilon}$ will admit a unique
absolutely continuous invariant measure with density $\rho_{\epsilon}.$ This
density will be related to the invariant density $\hat{\rho}_{\epsilon}$ of
the first return map $F_{\epsilon}$ on $I_{\epsilon}$ by the formula
\begin{equation}
\rho_{\epsilon}(x)=C_{\epsilon,r}\sum_{m=2}^{\infty}\sum_{l=1,2}\frac
{\hat{\rho}(T_{\epsilon,l}^{-1}T_{\epsilon,2}^{-1}T_{\epsilon,1}^{-(m-2)}%
x)}{|DT_{\epsilon}^{m}(T_{\epsilon,l}^{-1}T_{\epsilon,2}^{-1}T_{\epsilon
,1}^{-(m-2)}x)|}%
\end{equation}
for $\lambda$-a.e. $x\in(a_{n}^{\prime},a_{n-1}^{\prime})$ uniformly in
$n\geq1$ (see \cite{GMPV} formula (70)). Moreover, we will have that the
convergence of the perturbed map to the unperturbed one in the $C^{0}$
topology will imply that the density of the absolutely continuous invariant
perturbed measure converges to the density of the unperturbed measure in the
$L_{\lambda}^{1}$ norm.

\begin{figure}[ptbh]
\centering
\resizebox{0.75\textwidth}{!}{\includegraphics{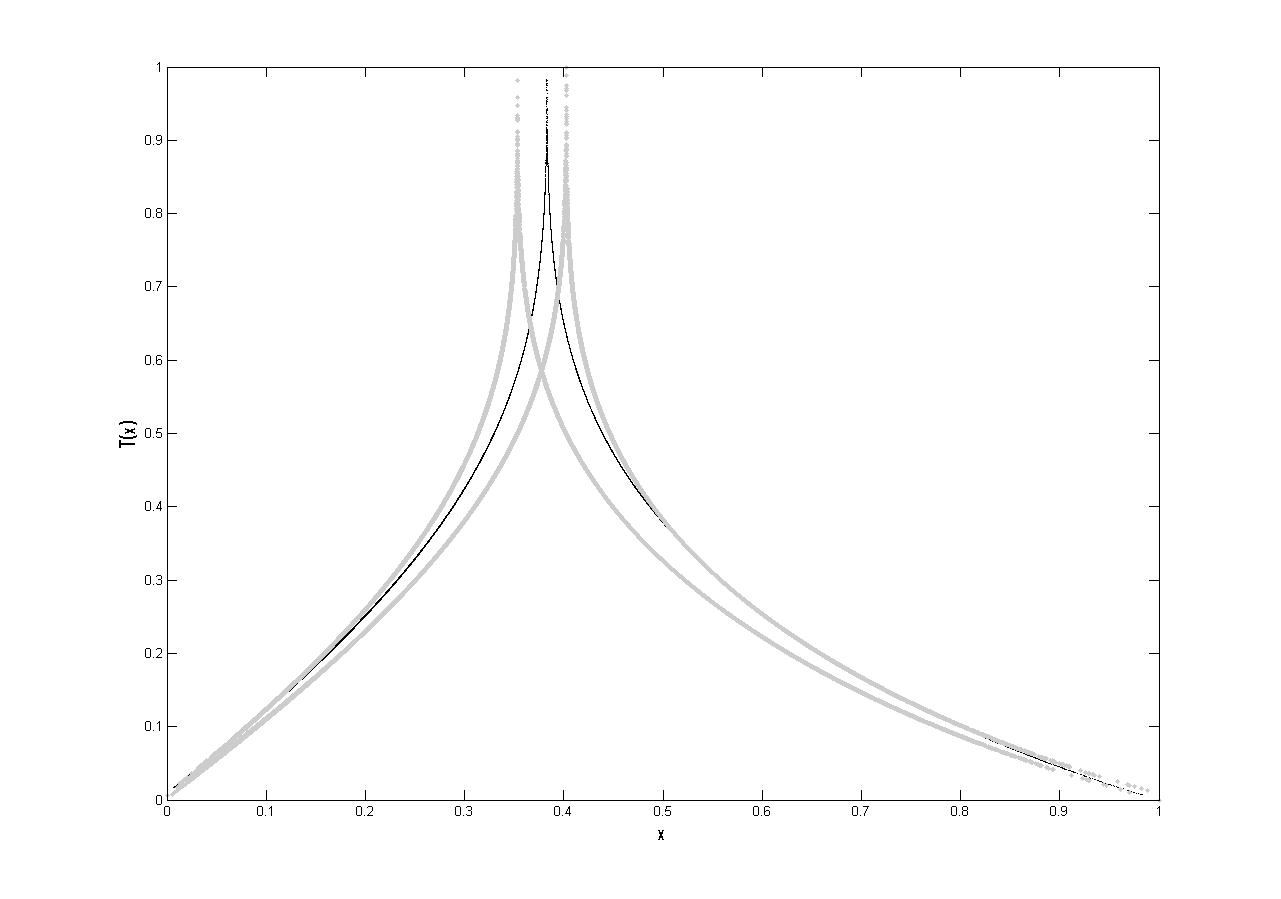}}
\caption{experimental plots of the unperturbed map $T$ (in black)
and of its perturbations (in grey).}%
\label{fig:2}%
\end{figure}

\begin{proposition}%
\begin{equation}
\lim_{\epsilon\rightarrow0^{+}}\left\Vert \rho-\rho_{\epsilon}\right\Vert
_{L_{\lambda}^{1}}=0\,.
\end{equation}

\end{proposition}

The proof of this result make use of induction but, in order to preserve the
Markov structure of the first return map, we need to compare the perturbed and
the unperturbed first return maps on different induction subsets. Therefore,
the difficulty in following this approach arises in the comparison of the
Perron-Frobenius operator associated to the induced perturbed system $\left(
I_{\epsilon},T_{I_{\epsilon}}\right)  $ with the Perron-Frobenius operator
associated with the unperturbed one $\left(  I,T_{I}\right)  ,$ which will now
be defined on different functional spaces. We defer the reader to \cite{GMPV}
for the details.

\section{Stochastic stability}

\subsection{Random perturbations}

Random perturbations of the classical Lorenz flow have been studied in the
framework of stochastic differential equations \cite{Sc}, \cite{CSG},
\cite{Ke} (see also \cite{Ar} and reference therein). The main interest of
these studies was bifurcation theory and the existence and the
characterization of the random attractor. The existence of the stationary
measure for this stochastic version of the system of equations given in
(\ref{L1}) is proved in \cite{Ke}.

Stochastic stability under diffusive type perturbations has been studied in
\cite{Ki} for the geometric Lorenz flow and in \cite{Me} for the contracting
Lorenz flow.

In \cite{GV} we introduced a random perturbation of the Lorenz'63 flow which,
being of impulsive nature, differ from diffusion-type perturbations.

\begin{itemize}
\item For any realization of the noise $\eta\in\left[  -\varepsilon
,\varepsilon\right]  ,$ we consider a flow $\left(  \Phi_{\eta}^{t}%
,t\geq0\right)  $ generated by the phase vector field $\phi_{\eta}$ belonging
to a sufficiently small neighborhood of the classical Lorenz one in the
$C^{1}$ topology.

\item For $\varepsilon$ small enough, the realizations of the perturbed phase
vector field $\phi_{\eta}$ can be chosen such that there exists an open
neighborhood $U$ of the unperturbed attractor in $\mathbb{R}^{3},$ independent
of the noise parameter $\eta,$ containing the attractor of any realization of
$\phi_{\eta}$

\item The perturbation acts modifying the phase velocity field of the system
at the ring of a random clock $\mathfrak{t}$
\end{itemize}

This procedure defines a semi-Markov random evolution (sMRE) \cite{KS}, in
fact a piecewise deterministic Markov process (PDMP) \cite{Da}.

To guarantee the existence of an invariant measure for a stochastic process of
this kind its imbedded renewal process must satisfy some minimal requirements.
In particular, when the evolution of the system is started outside $U$ the
trajectories of the system must enter in $U$ with probability one and when the
initial condition belongs to $U$ the expected number of modifications of the
phase vector field in a finite interval of time must be finite. Furthermore,
to make sure that the imbedded Markov chain has a stationary measure, it has
to satisfy some requirement such as for example to admit a Ljapunov function
(see e.g. \cite{MT}). Therefore, we assume that $\varepsilon$ sufficiently
small so that a given Poincar\'{e} section $\mathcal{M}$ for the unperturbed
flow is also transversal to any realization of the perturbed one and allow
changes in the phase velocity field of (\ref{L1}) just at the crossing of
$\mathcal{M}.$ Namely, let $\hat{\tau}_{\eta}:U\rightarrow\mathbb{R}$ and
$\tau_{\eta}:\mathcal{M}\rightarrow\mathbb{R}$ be respectively the hitting
time of $\mathcal{M}$ and the return time map on $\mathcal{M}$ for $\left(
\Phi_{\eta}^{t},t\geq0\right)  .$ If $\eta$ is sampled according to a given
law $\lambda_{\varepsilon}$ supported on $\left[  -\varepsilon,\varepsilon
\right]  ,$ the sequence $\left\{  \mathfrak{x}_{i}\right\}  _{i\geq0}$ such
that $\mathfrak{x}_{0}\in\mathcal{M}$ and, for $i\geq0,\mathfrak{x}%
_{i+1}:=\Phi_{\eta}^{\tau_{\eta}\left(  \mathfrak{x}_{i}\right)  }\left(
\mathfrak{x}_{i}\right)  $ is a homogeneous Markov chain on $\mathcal{M}$ with
transition probability measure
\begin{equation}
\mathbb{P}\left\{  \mathfrak{x}_{1}\in dz|\mathfrak{x}_{0}\right\}
=\lambda_{\varepsilon}\left\{  \eta\in\left[  -1,1\right]  :R_{\eta}\left(
\mathfrak{x}_{0}\right)  \in dz\right\}  \ .
\end{equation}
Considering the collection of sequences of i.i.d.r.v's $\left\{  \eta
_{i}\right\}  _{i\geq0}$ distributed according to $\lambda_{\varepsilon},$ we
define the random sequence $\left\{  \sigma_{n}\right\}  _{n\geq1}%
\in\mathbb{R}^{\mathbb{N}}$ such that $\sigma_{n}:=\sum_{i=0}^{n-1}\tau
_{\eta_{i-1}}\left(  \mathfrak{x}_{i-1}\right)  ,n\geq1.$ Then, it is easily
checked that the sequence $\left\{  \left(  \mathfrak{x}_{n},\mathbf{t}%
_{n}\right)  \right\}  _{n\geq0}$ such that $\mathbf{t}_{0}:=\sigma_{1}$ and,
for $n\geq0,\mathbf{t}_{n}:=\sigma_{n+1}-\sigma_{n}$ is a Markov renewal
process (MRP) \cite{As}, \cite{KS}. Therefore, denoting by $\left(
\mathbf{N}_{t},t\geq0\right)  ,$ such that $\mathbf{N}_{0}:=0$ and
$\mathbf{N}_{t}:=\sum_{n\geq0}\mathbf{1}_{\left[  0,t\right]  }\left(
\sigma_{n}\right)  ,$ the associated counting process and defining:

\begin{itemize}
\item $\left(  \mathfrak{x}_{t},t\geq0\right)  ,$ such that $\mathfrak{x}%
_{t}:=\mathfrak{x}_{\mathbf{N}_{t}},$ the associated semi-Markov process;

\item $\left(  \mathfrak{l}_{t},t\geq0\right)  ,$ such that $\mathfrak{l}%
_{t}:=t-\sigma_{\mathbf{N}_{t}},$ the \emph{age} (\emph{residual life}) of the MRP;

\item $\left(  \eta_{t},t\geq0\right)  $ such that $\eta_{t}:=\eta
_{\mathbf{N}_{t}},$
\end{itemize}

setting $\sigma_{0}:=\hat{\tau}_{\eta},$ we introduce the random process
$\left(  \mathfrak{u}_{t},t\geq0\right)  ,$ such that
\begin{equation}
\mathfrak{u}_{t}\left(  y_{0}\right)  :=\left\{
\begin{array}
[c]{ll}%
\Phi_{\eta}^{t}\left(  y_{0}\right)  \mathbf{1}_{[0,\sigma_{0}\left(
y_{0}\right)  )}\left(  t\right)  +\mathbf{1}_{\left\{  \Phi_{\eta}%
^{\sigma_{0}\left(  y_{0}\right)  }\left(  y_{0}\right)  \right\}  }\left(
\mathfrak{x}_{0}\right)  \Phi_{\eta_{\mathbf{N}_{t-\sigma_{0}\left(
y_{0}\right)  }}}^{\mathfrak{l}_{t-\sigma_{0}\left(  y_{0}\right)  }}%
\circ\mathfrak{x}_{t-\sigma_{0}\left(  y_{0}\right)  } & y_{0}\in
U\backslash\mathcal{M}\\
\mathbf{1}_{\left\{  y_{0}\right\}  }\left(  \mathfrak{x}_{0}\right)
\Phi_{\eta_{\mathbf{N}_{t}}}^{\mathfrak{l}_{t}}\circ\mathfrak{x}_{t} &
y_{0}\in\mathcal{M}%
\end{array}
\right.  \ , \label{u_t}%
\end{equation}
describes the system evolution started at $y_{0}\in U$ (Fig.3).

\begin{figure}[ptbh]
\centering
\resizebox{0.75\textwidth}{!}{\includegraphics{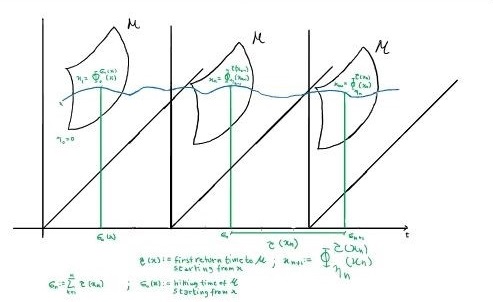}
}
\caption{perturbation scheme}%
\label{fig:3}%
\end{figure}

We prove:

\begin{theorem}
\label{main}There exists a measure $\mu_{\varepsilon}$ on the measurable space
$\left(  U,\mathcal{B}\left(  U\right)  \right)  ,$ with $\mathcal{B}\left(
U\right)  $ the trace $\sigma$algebra of the Borel $\sigma$algebra of
$\mathbb{R}^{3},$ such that, for any bounded real-valued measurable function
$f$ on $U,$%
\begin{equation}
\lim_{T\rightarrow\infty}\frac{1}{T}\int_{0}^{T}dtf\circ\mathfrak{u}_{t}%
=\mu_{\varepsilon}\left(  f\right)
\end{equation}
and
\begin{equation}
\lim_{\varepsilon\downarrow0}\mu_{\varepsilon}\left(  f\right)  =\mu
_{0}\left(  f\right)  \ ,
\end{equation}
where $\mu_{0}$ is the physical measure of the classical Lorenz flow.
\end{theorem}

In other words, we show that we can recover the physical measure of the
unperturbed flow as weak limit, as the intensity of the perturbation vanishes,
of the measure on the phase space of the system obtained by looking at the law
of large numbers for cumulative processes defined as the integral over
$\left[  0,t\right]  $ of functionals on the path space of the stationary
process representing the perturbed system's dynamics. Therefore, we will
reduce ourselves to prove that the imbedded Markov chain driving the random
process that describes the evolution of the system is stationary, that its
stationary (invariant) measure is unique and that it will converge weakly to
the invariant measure of the unperturbed Poincar\'{e} map corresponding to
$\mathcal{M}.$ To prove existence and uniqueness of the stationary initial
distribution of a Markov chain with uncountable state space is not an easy
task in general (we refer the reader to \cite{MT} for an account on this
subject). To overcome this difficulty we can take advantage of the
representation of the Markov chain $\left\{  \mathfrak{x}_{n}\right\}
_{n\geq0}$ as a Random Dynamical System (RDS) and consequently make use of the
skew-product structure of the first return maps $R_{\eta}.$ Furthermore, we
can also show that the trajectories of the PDMP (\ref{u_t}) are conjugated to
those of a suspension semi-flow over a RDS defined on $\left\{  \mathfrak{x}%
_{n}\right\}  _{n\geq0}$ with roof function defined in terms of the
realizations of the r.v. describing the return time on $\mathcal{M}.$ This led
us to give a proof of Theorem \ref{main} directly in the framework of the
theory of dynamical systems.

However, if the perturbation of the phase velocity field in (\ref{L1}) is
given by the addition to the unperturbed one $\phi_{0}$ of a small constant
term, namely $\phi_{\eta}:=\phi_{0}+\eta H,H\in\mathbb{S}^{2},$ the proof of
the existence of an invariant measure for the unperturbed Poincar\'{e} map
will follow a more direct strategy.

\subsubsection{Proof of Theorem \ref{main}}

The process $\left(  \mathfrak{v}_{t},t\geq0\right)  $ such that
$\mathfrak{v}_{t}:=\left(  \mathfrak{x}_{t},\mathbf{N}_{t},\mathfrak{l}%
_{t}\right)  $ is a homogeneous Markov process and so is the process $\left(
\mathfrak{w}_{t},t\geq0\right)  $ such that $\mathfrak{w}_{t}:=\left(
\mathfrak{x}_{t},\mathfrak{l}_{t}\right)  .$ Moreover $\overline{\mathcal{F}%
}_{t}^{\mathfrak{w}}\subseteq\overline{\mathcal{F}}_{t}^{\mathfrak{v}}$ and it
follows from \cite{Da} Theorem A2.2 that these $\sigma$algebras are both right continuous.

Assume for the moment that the Markov chain $\left\{  \mathfrak{x}%
_{n}\right\}  _{n\in\mathbb{Z}^{+}}$ is Harris recurrent (see e.g. \cite{MT}
or \cite{H-LL}). By formula (3.9) in \cite{Al} Corollary 1, (see also
\cite{Al} Theorem 3) we have that for any $x\in\mathcal{M},v\geq0$ and any
measurable set $A\subseteq\mathcal{M},$%
\begin{equation}
\lim_{t\rightarrow\infty}\mathbb{P}\left\{  \mathfrak{x}_{t}\in A,\mathfrak{l}%
_{t}>z|\mathfrak{x}_{0}=x,\mathfrak{l}_{0}=v\right\}  =\frac{\int
_{\mathcal{M}}\nu_{2}^{\varepsilon}\left(  dx\right)  \left[  \mathbf{1}%
_{A}\left(  x\right)  \int_{z}^{\infty}ds\left(  1-F_{\tau}^{\varepsilon
}\left(  s;x\right)  \right)  \right]  }{\int_{\mathcal{M}}\nu_{2}%
^{\varepsilon}\left(  dx\right)  \left[  \int_{0}^{\infty}ds\left(  1-F_{\tau
}^{\varepsilon}\left(  s;x\right)  \right)  \right]  }\;,\;\mathbb{P}%
\text{-a.s.\ ,}%
\end{equation}
where for any $x\in\mathcal{M},t\geq0,$%
\begin{equation}
F_{\tau}^{\varepsilon}\left(  t;x\right)  :=\mathbb{P}\left\{  \omega\in
\Omega:\mathbf{t}\left(  x,\omega\right)  \leq t\right\}  =\lambda
_{\varepsilon}\left\{  \eta\in\left[  -1,1\right]  :\tau_{\eta}\left(
x\right)  \leq t\right\}
\end{equation}
and $\nu_{2}^{\varepsilon}$ is the stationary measure for $\left\{
\mathfrak{x}_{n}\right\}  _{n\in\mathbb{Z}^{+}}.$ The following result, the
proof of which is deferred to \cite{GV}, defines the physical measure of
(\ref{u_t}).

\begin{proposition}
For any bounded measurable function $f$ on $U$ and any $y_{0}\in U,$%
\begin{equation}
\lim_{t\rightarrow\infty}\frac{1}{t}\int_{0}^{t}dsf\circ\mathfrak{u}%
_{s}\left(  y_{0}\right)  =\frac{\int_{\left[  -1,1\right]  }\lambda
_{\varepsilon}\left(  d\eta\right)  \int_{\mathcal{M}}\nu_{2}^{\varepsilon
}\left(  dx\right)  \int_{0}^{\tau_{\eta}\left(  x\right)  }dsf\left(
\Phi_{\eta}^{s}\left(  x\right)  \right)  }{\int_{\mathcal{M}}\nu
_{2}^{\varepsilon}\left(  dx\right)  \left[  \int_{0}^{\infty}ds\left(
1-F_{\tau}^{\varepsilon}\left(  s;x\right)  \right)  \right]  }\;,\;\mathbb{P}%
\text{-a.s.}%
\end{equation}

\end{proposition}

Therefore, defining
\begin{equation}
\mu_{\varepsilon}\left(  f\right)  :=\frac{\int_{\left[  -1,1\right]  }%
\lambda_{\varepsilon}\left(  d\eta\right)  \int_{\mathcal{M}}\nu
_{2}^{\varepsilon}\left(  dx\right)  \int_{0}^{\tau_{\eta}\left(  x\right)
}ds}{\int\nu_{2}^{\varepsilon}\left(  dx\right)  \left[  \int_{0}^{\infty
}ds\left(  1-F_{\tau}^{\varepsilon}\left(  s;x\right)  \right)  \right]
}f\circ\Phi_{\eta}^{s}\left(  x\right)  \ ,
\end{equation}
assuming the stochastic stability of the invariant measure $\mu_{R_{0}}$ for
the unperturbed Poincar\'{e} map $R_{0},$ namely the weak convergence of
$\nu_{2}^{\varepsilon}$ to $\mu_{R_{0}}$ as $\varepsilon\downarrow0,$ since
for any bounded real-valued measurable function $\varphi$ on $\mathcal{M}%
\times\mathbb{R}^{+},$%
\begin{gather}
\lim_{\varepsilon\rightarrow0}\frac{1}{\nu_{2}^{\varepsilon}\left[  \int
_{0}^{\infty}ds\left(  1-F_{\tau}^{\varepsilon}\left(  s;\cdot\right)
\right)  \right]  }\int_{\mathcal{M}}\nu_{2}^{\varepsilon}\left(  dx\right)
\int_{0}^{\tau_{\eta}\left(  x\right)  }ds\varphi\left(  x,s\right)  =\\
=\int_{\mathcal{M}}\mu_{R_{0}}\left(  dx\right)  \int_{0}^{\tau_{0}\left(
x\right)  }ds\frac{1}{\mu_{R_{0}}\left[  \tau_{0}\right]  }\varphi\left(
x,s\right)  =\mu_{S_{0}}\left(  \varphi\right)  \ ,\nonumber
\end{gather}
where $S_{0}$ is the suspension flow over $R_{0}$ with roof function $\tau
_{0},$ we get
\begin{equation}
\lim_{\varepsilon\rightarrow0}\mu_{\varepsilon}\left(  f\right)  =\mu_{S_{0}%
}\left(  f\circ\Phi_{0}^{\cdot}\right)  =\int_{\mathcal{M}}\mu_{R_{0}}\left(
dx\right)  \int_{0}^{\tau_{0}\left(  x\right)  }ds\frac{1}{\mu_{R_{0}}\left[
\tau_{0}\right]  }f\circ\Phi_{0}^{s}\left(  x\right)  \ ,
\end{equation}
that is the proof of the following result.

\begin{theorem}
If $\nu_{2}^{\varepsilon}$ weakly converges to $\mu_{R_{0}},$ then
$\mu_{\varepsilon}$ weakly converges to the unperturbed physical measure.
\end{theorem}

Therefore we are left with the proof of the existence and uniqueness of
$\nu_{2}^{\varepsilon}$ and of its weak convergence to $\mu_{R_{0}}$ in the
limit $\varepsilon\downarrow0.$ As we have already outlined, this can be done
making use of the representation of the Markov chain $\left\{  \mathfrak{x}%
_{n}\right\}  _{n\in\mathbb{Z}^{+}}$ as RDS.

\begin{remark}
We remark that the Harris recurrence property, which once proven to hold for
$\left\{  \mathfrak{x}_{n}\right\}  _{n\in\mathbb{Z}^{+}}$ entail the
existence and uniqueness of $\nu_{2}^{\varepsilon},$ will hold for the Markov
chain described in step 2 below, since by construction (see assumption A5
below) its transition probabilities satisfy condition (i) of Theorem 3.1 in
\cite{H-LL}. This in turn will imply the SLLN for sequences of r.v.'s of the
form $\left\{  f\circ\mathfrak{x}_{n}\right\}  _{n\in\mathbb{Z}^{+}},$ for any
$f\in L^{1}\left(  \nu_{2}^{\varepsilon}\right)  ,$ hence the Harris
recurrence property for $\left\{  \mathfrak{x}_{n}\right\}  _{n\in
\mathbb{Z}^{+}},$ by Remark 3 and Corollary 4, in view of Proposition 2, in
\cite{GV}.

We also remark that in the proof of Theorem \ref{main} carried on by following
the steps 1 to 8 listed below we do not need to take in to account the Harris
recurrence property of the driving Markov chain $\left\{  \mathfrak{x}%
_{n}\right\}  _{n\in\mathbb{Z}^{+}}$ of the PDMP (\ref{u_t}).

On the other hand, proceeding as in \cite{Op} is possible to prove the
invariance principle (functional CLT) and almost sure invariance principle
(almost sure functional CLT) for a class of additive functionals of the
semi-Markov process $\left(  \mathfrak{x}_{t},t\geq0\right)  $ and as a direct
consequence for a class of additive functionals of the PDMP (\ref{u_t}). We
stress that, in our case, assumption A2 in \cite{Op} can be replaced by the
requirement of the existence and uniqueness of $\nu_{2}^{\varepsilon}.$
\end{remark}

In the special case of random perturbations of $\left(  \Phi_{0}^{t}%
,t\geq0\right)  $ realized by the addition to the unperturbed phase vector
field of a constant random term, namely
\begin{equation}
\phi_{\eta}:=\phi_{0}+\eta H\ ,\;\eta\in spt\lambda_{\varepsilon}\ ,
\end{equation}
with $H\in\mathbb{R}^{3},$ we can take a step forward w.r.t. the problem of
showing the existence of an invariant measure for $\left\{  \mathfrak{x}%
_{n}\right\}  _{n\in\mathbb{Z}^{+}}.$\ Indeed, it has been shown in \cite{PP}
that the Casimir function defined in section \ref{statstab} is a Lyapunov
function for the ODE system defined by $\phi_{\eta},$ namely, for any
realization of the noise $\eta\in spt\lambda_{\varepsilon},$
\[
\left(  C\circ\Phi_{\eta}^{t}\right)  \left(  u\right)  \leq C\left(
u\right)  e^{-t\min\left(  1,\zeta,\beta\right)  }+\frac{\left\Vert H_{\eta
}\right\Vert ^{2}}{\left(  \min\left(  1,\zeta,\beta\right)  \right)  ^{2}%
}\left(  1+e^{-t\min\left(  1,\zeta,\beta\right)  }\right)  \ ,
\]
where $\mathbb{R}^{3}\ni u\longmapsto C\left(  u\right)  :=\left\langle
u,u\right\rangle =\left\Vert u\right\Vert ^{2}\in\mathbb{R}^{+}$ and $H_{\eta
}:=\eta H+H_{0}\in\mathbb{R}^{3},$ with $H_{0}:=\left(  0,0,-\beta\left(
\zeta+\gamma\right)  \right)  .$

Hence, choosing $t=\tau_{\eta}\left(  u\right)  $ we obtain
\begin{equation}
C\circ R_{\eta}\left(  u\right)  \leq a_{\varepsilon}C\left(  u\right)
+K_{\varepsilon}\left(  1+a_{\varepsilon}\right)  \ ,
\end{equation}
where
\begin{align}
a_{\varepsilon}  &  :=e^{-\min\left(  1,\zeta,\beta\right)  \inf_{\eta\in
spt\lambda_{\varepsilon}}\inf_{u\in\mathcal{M}}\tau_{\eta}\left(  u\right)
}\in\left(  0,1\right)  \;,\\
K_{\varepsilon}  &  :=\frac{\sup_{\eta\in spt\lambda_{\varepsilon}}\left\Vert
H_{\eta}\right\Vert ^{2}}{\left(  \min\left(  1,\zeta,\beta\right)  \right)
^{2}}>0\;.
\end{align}
Moreover, for any $\varsigma>0,$
\begin{align}
\left(  1+\varsigma C\right)  \circ R_{\eta}\left(  u\right)   &
\leq1+\varsigma a_{\varepsilon}C\left(  u\right)  +\varsigma K_{\varepsilon
}\left(  1+a_{\varepsilon}\right) \label{LY1}\\
&  =a_{\varepsilon}\left(  1+\varsigma C\left(  u\right)  \right)  +\bar
{K}_{\varepsilon}\ ,\nonumber
\end{align}
where $\bar{K}_{\varepsilon}:=\left(  1-a_{\varepsilon}\right)  +\varsigma
K_{\varepsilon}\left(  1+a_{\varepsilon}\right)  ,$ which entails for the
transition operator $P_{R}$ of the Markov chain $\left\{  \mathfrak{x}%
_{n}\right\}  _{n\in\mathbb{Z}^{+}}$%
\begin{equation}
C_{b}\left(  \mathcal{M}\right)  \ni\psi\longmapsto P_{R}\psi\in M_{b}\left(
\mathcal{M}\right)  \ ,
\end{equation}
the \emph{weak drift condition}
\begin{equation}
P_{R}\left(  1+\varsigma C\right)  \left(  u\right)  \leq a_{\varepsilon
}\left(  1+\varsigma C\left(  u\right)  \right)  +\bar{K}_{\varepsilon}\ ,
\label{wd}%
\end{equation}
which implies the following

\begin{lemma}
$P_{R}$ admits an invariant probability measure.
\end{lemma}

\begin{proof}
Let $\mathbb{B}_{0}$ be the dual space of $C\left(  \mathcal{M}\right)  $ and
$\mathbb{B}_{\varsigma}$ be the dual space of $C_{\varsigma}\left(
\mathcal{M}\right)  $: the Banach space of real-valued functions on
$\mathcal{M}$ such that $\sup_{x\in\mathcal{M}}\frac{\left\vert \psi\left(
x\right)  \right\vert }{1+\varsigma C\left(  x\right)  }<\infty.\mathbb{B}%
_{\varsigma}\subseteq\mathbb{B}_{0}$ and (\ref{LY1}), (\ref{wd}) are
respectively equivalent to the Doeblin-Fortet conditions, namely, for any
$\mu\in\mathbb{B}_{\varsigma}$%
\begin{align}
\left\Vert \left(  R_{\eta}\right)  _{\#}\mu\right\Vert _{\varsigma}  &  \leq
a_{\varepsilon}\left\Vert \mu\right\Vert _{\varsigma}+\bar{K}_{\varepsilon
}\left\Vert \mu\right\Vert _{0}\ ,\\
\left\Vert \mu P_{R}\right\Vert _{\varsigma}  &  \leq a_{\varepsilon
}\left\Vert \mu\right\Vert _{\varsigma}+\bar{K}_{\varepsilon}\left\Vert
\mu\right\Vert _{0}\ , \label{wd1}%
\end{align}
where $\left\Vert \cdot\right\Vert _{0},\left\Vert \cdot\right\Vert
_{\varsigma}$ denote the norm of $\mathbb{B}_{0}$ and $\mathbb{B}_{\varsigma
}.$ This, together with the tightness of $\mathbb{B}_{0}$ due to the
compactness of $\mathcal{M}$ imply the thesis (see \cite{GV} Lemma 25).
\end{proof}

The proof of the existence and uniqueness of $\nu_{2}^{\varepsilon}$ (simply
of uniqueness in the case of the additive perturbation scheme just described)
rely on the proof of the existence and uniqueness of the stationary measure
$\nu_{1}^{\varepsilon}$ of the RDS describing the random perturbations of the
one-dimensional quotient map representing the evolution of the leaf of the
invariant foliation of $\mathcal{M}$ introduced at the beginning of section
\ref{statstab}.

In order to simplify the exposition, which contains many technical details and
requires the introduction of several quantities, we will list here the main
steps we will go through to get to the proof deferring the reader to \cite{GV}
part II for a detailed and precise description.

\begin{itemize}
\item[\textbf{Step 1}] For any $\eta\in\left[  -\varepsilon,\varepsilon
\right]  ,$ the perturbed phase field $\phi_{\eta}$ is such that the
associated flows $\left(  \Phi_{\eta}^{t},t\geq0\right)  $ admit a $C^{1}$
stable foliation in a neighborhood of the corresponding attractor. In order to
study the RDS defined by the composition of the maps $R_{\eta}:=\Phi_{\eta
}^{\tau_{\eta}}:\mathcal{M}\circlearrowleft,$ with $\tau_{\eta}:\mathcal{M}%
\circlearrowleft$ the return time map on $\mathcal{M}$ for $\left(  \Phi
_{\eta}^{t},t\geq0\right)  ,$ we show that we can restrict ourselves to study
a RDS given by the composition of maps $\bar{R}_{\eta}:\mathcal{M}%
\circlearrowleft,$ conjugated to the maps $R_{\eta}$ via a diffeomorphism
$\kappa_{\eta}:\mathcal{M}\circlearrowleft,$ leaving invariant the unperturbed
stable foliation for any realization of the noise. Namely, we can reduce the
cross-section to a unit square foliated by vertical stable leaves, as for the
geometric Lorenz flow. By collapsing these leaves on their base points via the
diffeomorphism $q,$ we conjugate the first return map $\bar{R}_{\eta}$ on
$\mathcal{M}$ to a piecewise map $\bar{T}_{\eta}$ of the interval $I.$ This
one-dimensional quotient map is expanding with the first derivative blowing up
to infinity at some point.

\item[\textbf{Step 2}] We introduce the random perturbations of the
unperturbed quotient map $T_{0}.$ Suppose $\omega=(\eta_{0},\eta_{1}%
,\cdots,\eta_{k},\cdots)$ is a sequence of values in $\left[  -\varepsilon
,\varepsilon\right]  $ each chosen independently of the others according to
the probability $\lambda_{\varepsilon}.$ We construct the concatenation
$\bar{T}_{\eta_{k}}\circ\cdots\circ\bar{T}_{\eta_{0}}$ and prove that there
exists a stationary measure $\nu_{1}^{\varepsilon},$ i.e. such that for any
bounded measurable function $g$ and $k\geq0,\int g(\bar{T}_{\eta_{k}}%
\circ\cdots\circ\bar{T}_{\eta_{0}})(x)\nu_{1}^{\varepsilon}\left(  dx\right)
\lambda_{\varepsilon}^{\otimes k}(d\eta)=\int gd\nu_{1}^{\varepsilon}.$
Clearly, $\mu_{\mathbf{T}}^{\varepsilon}:=\nu_{1}^{\varepsilon}\otimes
\mathbb{P}_{\varepsilon},$ with $\mathbb{P}_{\varepsilon}$ the probability
measure on the i.i.d. random sequences $\omega,$ is an invariant measure for
the associated RDS (see \cite{GV} formula (46)).

\item[\textbf{Step 3}] We lift the random process just defined to a Markov
process on the Poincar\'{e} surface $\mathcal{M}$ given by the sequences
$\bar{R}_{\eta_{k}}\circ\cdots\circ\bar{R}_{\eta_{0}}$ and show that the
stationary measure $\nu_{2}^{\varepsilon}$ for this process can be constructed
from $\nu_{1}^{\varepsilon}.$ We set $\mu_{\overline{\mathbf{R}}}%
^{\varepsilon}:=\bar{\nu}_{2}^{\varepsilon}\otimes\mathbb{P}_{\varepsilon}$
the corresponding invariant measure for the RDS (see \cite{GV} formula (47)).

We remark that, by construction, the conjugation property linking $R_{\eta}$
with $\bar{R}_{\eta}$ lifts to the associated RDS's. This allows us to recover
from $\mu_{\overline{\mathbf{R}}}^{\varepsilon}$ the invariant measure
$\mu_{\mathbf{R}}^{\varepsilon}$ for the RDS generated by composing the
$R_{\eta}$'s.

\item[\textbf{Step 4}] Let $\mathbf{R}:\mathcal{M}\times\Omega\circlearrowleft
$ be the map defining the RDS corresponding to the compositions of the
realizations of $R_{\eta}$ (see \cite{GV} formula (52)). We identify the set
\begin{equation}
(\mathcal{M}\times\Omega)_{\mathbf{t}}:=\{(x,\omega,s)\in\mathcal{M}%
\times\Omega\times\mathbb{R}^{+}:s\in\lbrack0,\mathbf{t}(x,\omega))\}\ ,
\end{equation}
where $\Omega:=\left[  -\varepsilon,\varepsilon\right]  ^{\mathbb{N}%
},\mathbf{t}(x,\omega):=\tau_{\pi(\omega)}(x)$ is the \emph{random roof
function} and $\pi(\omega):=\eta_{0}$ is the first coordinate of $\omega, $
with the set $\mathfrak{V}$ of equivalence classes of points $\left(
x,\omega,t\right)  $ in $\mathcal{M}\times\Omega\times\mathbb{R}^{+}$ such
that $t=s+\sum_{k=0}^{n-1}\mathbf{t}\left(  \mathbf{R}^{k}\left(
x,\omega\right)  \right)  $ for some $s\in\lbrack0,\mathbf{t}(x,\omega
)),n\geq1.$ Then, if $\hat{\pi}:\mathcal{M}\times\Omega\times\mathbb{R}%
^{+}\longrightarrow\mathfrak{V}$ is the canonical projection and, for any
$t>0,N_{t}:=\max\left\{  n\in\mathbb{Z}^{+}:\sum_{k=0}^{n-1}\mathbf{t}%
\circ\mathbf{R}^{k}\leq t\right\}  ,$ we define the \emph{random suspension
semi-flow}
\begin{equation}
(\mathcal{M}\times\Omega)_{\mathbf{t}}\ni\left(  x,\omega,s\right)
\longmapsto\mathbf{S}^{t}(x,\omega,s):=\hat{\pi}(\mathbf{R}^{N_{s+t}}\left(
x,\omega\right)  ,s+t)\in(\mathcal{M}\times\Omega)_{\mathbf{t}}\ .
\end{equation}
In particular, for instance, if $\mathbf{s}_{2}(x,\omega)=\tau_{\eta_{0}%
}(x)+\tau_{\eta_{1}}(R_{\eta_{1}}(x))\leq s+t,$ we have
\begin{equation}
\mathbf{S}^{t}(x,\omega,s)=((R_{\eta_{1}}\circ R_{\eta_{0}}(x)),\theta
^{2}\omega,s+t-\mathbf{s}_{2}(x,\omega))\ ,
\end{equation}
where $\theta:\Omega\ni\omega=(\eta_{0},\eta_{1},\cdots,\eta_{k}%
,\cdots)\longmapsto\theta\omega:=(\eta_{1},\eta_{2},\cdots,\eta_{k+1}%
,\cdots)\in\Omega$ is the left shift.

\item[\textbf{Step 5}] We build up a conjugation between the random suspension
semi-flow and a semi-flow on $U\times\Omega,$ which we will call $\left(
X^{t},t\geq0\right)  ,$ such that its projection on $U$ is a representation of
(\ref{u_t}). The rough idea is that each time the orbit crosses the
Poincar\'{e} section $\mathcal{M},$ the vector fields will change randomly.
Therefore, we start by fixing the \emph{initial condition} $(y,\omega)$ with
$y\in U$ yet not necessarily on $\mathcal{M}.$ We now begin to define the
\emph{random flow} $\left(  X^{t},t\geq0\right)  .$ Let $\pi:\Omega
\mapsto\left[  -\varepsilon,\varepsilon\right]  $ be the projection of
$\omega=(\eta_{0},\eta_{1},\cdots,\eta_{k},\cdots)$ onto the first coordinate
and call $t_{\eta_{0}}\left(  y\right)  =t_{\pi\left(  \omega\right)  }\left(
y\right)  $ the time the orbit $\Phi_{\eta_{0}}^{t}(y)=\Phi_{\pi\left(
\omega\right)  }^{t}(y)$ takes to meet $\mathcal{M}$ and set $y_{1}%
:=\Phi_{\eta_{0}}^{t_{\eta_{0}}\left(  y\right)  }(y)=\Phi_{\pi\left(
\omega\right)  }^{t_{\pi\left(  \omega\right)  }\left(  y\right)  }(y).$Then,
since $\forall\omega\in\Omega,n\geq0,\pi\left(  \theta^{n}\omega\right)
=\eta_{n},$%
\begin{align}
X^{t}(y,\omega) &  :=\left(  \Phi_{\pi\left(  \omega\right)  }^{t}%
(y),\omega\right)  ,\ 0\leq t\leq t_{\eta_{0}}\left(  y\right)  \ ;\label{X_t}%
\\
X^{t}(y,\omega) &  =\left(  \Phi_{\pi\left(  \theta\omega\right)  }%
^{t-t_{\pi\left(  \omega\right)  }\left(  y\right)  }(y_{1}),\theta
\omega\right)  ,\ t_{\eta_{0}}\left(  y\right)  <t\leq t_{\eta_{0}}\left(
y\right)  +\tau_{\eta_{1}}(y_{1})\ ;\nonumber\\
X^{t}(y,\omega) &  =\left(  \Phi_{\pi\left(  \theta^{2}\omega\right)
}^{t-t_{\pi\left(  \omega\right)  }\left(  y\right)  -\tau_{\pi\left(
\theta\omega\right)  }(y_{1})}(R_{\pi\left(  \theta\omega\right)  }%
(y_{1})),\theta^{2}\omega\right)  ,\ \left.
\begin{array}
[c]{l}%
t>t_{\eta_{0}}\left(  y\right)  +\tau_{\eta_{1}}(y_{1})\\
t\leq t_{\eta_{0}}\left(  y\right)  +\tau_{\eta_{1}}(y_{1})+\tau_{\eta_{2}%
}(R_{\eta_{1}}(y_{1}))
\end{array}
\right.  \ ,\nonumber
\end{align}
where $R_{\pi\left(  \theta\omega\right)  }(y_{1})=R_{\eta_{1}}(y_{1}),$ and
so on.

\item[\textbf{Step 6}] We are now ready to define the conjugation
$\mathbf{V}:\mathcal{M}\times\Omega\times\mathbb{R}^{+}\rightarrow
\mathbb{R}^{3}\times\Omega$ in the following way:
\begin{align}
\mathbf{V}(x,\omega,s) &  =\left(  \Phi_{\pi\left(  \omega\right)  }%
^{s}(x),\omega\right)  ,\ x\in\mathcal{M};\ \omega=(\eta_{0},\eta_{1}%
,\cdots,\eta_{k},\cdots)\in\Omega;\ 0\leq s<\tau_{\eta_{0}}(x)\\
\mathbf{V}(x,\omega,s) &  =\left(  \Phi_{\pi\left(  \theta\omega\right)
}^{s-\tau_{\pi\left(  \omega\right)  }(x)}(R_{\pi\left(  \omega\right)
}(x)),\theta\omega\right)  ;\ \tau_{\eta_{0}}(x)\leq s<\tau_{\eta_{0}}%
(x)+\tau_{\eta_{1}}(R_{\eta_{0}}(x))\ ,\nonumber
\end{align}
where $R_{\pi\left(  \omega\right)  }(x)=R_{\eta_{0}}(x),$ and so on. By
collecting the expressions given above it is not difficult to check that
$\left(  X^{t},t\geq0\right)  $ must satisfy the equation
\begin{equation}
\mathbf{V}\circ\mathbf{S}^{t}=X^{t}\circ\mathbf{V}\ .
\end{equation}
For instance, if $s+t<\tau_{\eta_{0}}(x),$ we have $X^{t}\circ\mathbf{V}%
(x,\omega,s)=\left(  X^{t}(\Phi_{\eta_{0}}^{s}(x)),\omega\right)  =\left(
\Phi_{\eta_{0}}^{t}(\Phi_{\eta_{0}}^{s}(x)),\omega\right)  =\left(  \Phi
_{\eta_{0}}^{s+t}(x),\omega\right)  ,$ while $\mathbf{V}\circ\mathbf{S}%
^{t}(x,\omega,s)=\mathbf{V}(x,\omega,s+t)=\left(  \Phi_{\eta_{0}}%
^{s+t}(x),\omega\right)  .$

\item[\textbf{Step 7}] We lift the measure $\mu_{\mathbf{R}}^{\varepsilon}$ on
the random suspension in order to get an invariant measure for $\left(
\mathbf{S}^{t},t\geq0\right)  .$ Under the assumption that the random roof
function $\mathbf{t}$ is $\mu_{\mathbf{R}}^{\varepsilon}$-summable, the
invariant measure $\mu_{\mathbf{S}}^{\varepsilon}$ for the random suspension
semi-flow acts on bounded real functions $f$ as
\begin{equation}
\int d\mu_{\mathbf{S}}^{\varepsilon}f=\left(  \int d\mu_{\mathbf{R}%
}^{\varepsilon}\mathbf{t}\right)  ^{-1}\int d\mu_{\mathbf{R}}^{\varepsilon
}\left(  \int_{0}^{\mathbf{t}}f\circ\mathbf{S}^{t}dt\right)  \ .
\end{equation}

The invariant measure for the random flow $\left(  X^{t},t\geq0\right)  $ will
then be push forward $\mu_{\mathbf{S}}^{\varepsilon}$ under the conjugacy
$\mathbf{V},$ i.e.%
\begin{equation}
\mu_{\mathbf{V}}^{\varepsilon}=\mu_{\mathbf{S}}^{\varepsilon}\circ
\mathbf{V}^{-1}\ .
\end{equation}

\item[\textbf{Step 8}] We show that the correspondence $\mu_{\mathbf{T}%
}^{\varepsilon}\longrightarrow\mu_{\mathbf{R}}^{\varepsilon}\longrightarrow
\mu_{\mathbf{V}}^{\varepsilon}$ is injective and so that the stochastic
stability of $T_{0}$ (which in fact we prove to hold in the $L^{1}\left(
I,dx\right)  $ topology) implies that of the physical measure $\mu_{0}$ of the
unperturbed flow. More precisely, we lift the evolutions defined by the
unperturbed maps $T_{0}$ and $R_{0},$ as well as that represented by the
unperturbed suspension semi-flow $\left(  S_{0}^{t},t\geq0\right)  ,$ to
evolutions defined respectively on $I\times\Omega,\mathcal{M}\times\Omega$ and
on $(\mathcal{M}\times\Omega)_{\tau_{0}}:=\{(x,\omega,s)\in\mathcal{M}%
\times\Omega\times\mathbb{R}^{+}:s\in\lbrack0,\tau_{0}(x))\}.$ By
construction, the invariant measures for these evolutions are $\mu_{T_{0}%
}\otimes\delta_{\bar{0}},\mu_{R_{0}}\otimes\delta_{\bar{0}},\mu_{S_{0}}%
\otimes\delta_{\bar{0}},$ where $\bar{0}$ denotes the sequence in $\Omega$
whose entries are all equal to $0,\delta_{\bar{0}}$ is the Dirac mass at
$\bar{0}$ and $\mu_{T_{0}},\mu_{R_{0}},\mu_{S_{0}}$ are respectively the
invariant measures for $T_{0},R_{0}$ and $S_{0}.$ Then, we prove the weak
convergence, as $\varepsilon\downarrow0,$ of $\mu_{\mathbf{T}}^{\varepsilon}$
to $\mu_{T_{0}}\otimes\delta_{\bar{0}}$ and consequently the weak convergence
of $\mu_{\mathbf{R}}^{\varepsilon}$ to $\mu_{T_{0}}\otimes\delta_{\bar{0}}.$
This will imply the weak convergence of $\mu_{\mathbf{S}}^{\varepsilon}$ to
$\mu_{S_{0}}\otimes\delta_{\bar{0}}$ and therefore the weak convergence of
$\mu_{\mathbf{V}}^{\varepsilon}$ to $\mu_{0}$ providing another proof to
Theorem \ref{main}.
\end{itemize}

We are then left with the proof of the stochastic stability of $\mu_{T_{0}}.$
Here we report a brief account on this subject referring the reader to section
8.4 in \cite{GV} for a more detailed description.

We denote by $\mathcal{L}$ the transfer operator of the unperturbed map
$T:=T_{0},$ by $\mathcal{L}_{\varepsilon}$ the random transfer operator
defined by the formula $\mathcal{L}_{\varepsilon}(f)=\int_{\left[
-1,1\right]  }d\lambda_{\varepsilon}\left(  \eta\right)  \mathcal{L}_{\eta}f,$
where $f$ belongs to some Banach space $\mathbb{B}\subset L^{1}$ and by
$\mathcal{L}_{\eta}$ the transfer operator associated to the perturbed map
$T_{\eta}.$ Let us suppose that:

\begin{description}
\item[A1] The unperturbed transfer operator $\mathcal{L}$ verifies the
so-called Lasota-Yorke inequality, namely there exists constants
$0<\varkappa<1,D>0,$ such that for any $f\in\mathbb{B}$ we have
\[
\left\Vert \mathcal{L}f\right\Vert _{\mathbb{B}}\leq\varkappa\left\Vert
f\right\Vert _{\mathbb{B}}+D\left\Vert f\right\Vert _{1}\ .
\]

\item[A2] The map $T$ preserve only one absolutely continuous invariant
probability measure $\mu$ with density $h,$ which therefore will be also
ergodic and mixing.

\item[A3] The random transfer operator $\mathcal{L}_{\varepsilon}$ verifies a
similar Lasota-Yorke inequality which, for sake of simplicity, we will assume
to hold with the same parameters $\varkappa$ and $D.$

\item[A4] There exits a measurable function $\left[  -1,1\right]
\ni\varepsilon\longmapsto\upsilon^{\prime}(\varepsilon)\in\mathbb{R}^{+}$
tending to zero when $\varepsilon\rightarrow0$ such that for $f\in
\mathbb{B}:$
\[
|||\mathcal{L}f-\mathcal{L}_{\varepsilon}f|||\leq\upsilon^{\prime}%
(\varepsilon).
\]
where the norm $|||\cdot|||$ above is so defined: $|||L|||:=\sup_{\left\Vert
f\right\Vert _{\mathbb{B}}\leq1}\left\Vert Lf\right\Vert _{1},$ for a linear
operator $L:L^{1}\circlearrowleft.$

\item[A5] The transition probability $\mathcal{Q}(x,A)$ admits a density
$\mathfrak{q}_{\varepsilon}(x,y),$ namely: $\mathcal{Q}(x,A)=\int
_{A}\mathfrak{q}_{\varepsilon}(x,y)dy;$

\item[A6] $spt\mathcal{Q}(x,\cdot)=B_{\varepsilon}(Tx),$ for any $x$ in the
interval, where $B_{\varepsilon}(z)$ denotes the ball of center $z$ and radius
$\varepsilon.$
\end{description}

By \cite{BHV} assumptions A1 - A3, A5 - A6 guarantee that there will be only
one absolutely continuous stationary measure $\nu_{1}^{\varepsilon}$ with
density $h_{\varepsilon}$ for the Markov chain with transition operator
associated to $\mathcal{L}_{\varepsilon}.$ Assumption A4 allow us to invoke
the perturbation theorem of Keller and Liverani \cite{KL} to assert that the
norm $|||\cdot|||$ of the difference of the spectral projections of the
operators $\mathcal{L}$ and $\mathcal{L}_{\varepsilon}$ associated with the
eigenvalue $1$ goes to zero when $\varepsilon\downarrow0.$ Since the
corresponding eigenspace have dimension $1,$ we conclude that $h_{\varepsilon
}\rightarrow h$ in the $L^{1}$ norm and we have proved the stochastic
stability in the strong sense.

Remember that $spt\lambda_{\varepsilon}\subset(-\varepsilon,\varepsilon)$ and
choose the maps $T_{\eta}$ with absolutely continuous invariant distribution
$\mu_{\eta}$ in such a way they are close to $T$ in the following sense:

\begin{itemize}
\item denoting by $g=\frac{1}{|T^{\prime}|}$ and $g_{\eta}=\frac{1}{|T_{\eta
}^{\prime}|}$ the potentials of the two maps defined everywhere but in the
discontinuity, or critical, points $x_{0}$ and $x_{0,\eta}$ respectively, we
have that $g$ and $g_{\eta}$ satisfy the H\"{o}lder conditions, with the same
constant and exponent (we can always reduce to this case by choosing
$\varepsilon$ sufficiently small):
\[
|g(x)-g(y)|\leq C_{h}|x-y|^{\epsilon}\ ;\ |g_{\eta}(x)-g_{\eta}(y)|\leq
C_{h}|x-y|^{\epsilon}\ ,
\]
where $(x,y)$ belong to the two domains on injectivity of the maps excluding
the critical points. We will call these domains $I_{1},I_{2}$ and $I_{1,\eta
},I_{2,\eta}$ respectively assuming that the domain labelled with $i=1$ is the leftmost.

\item The branches are \emph{horizontally close}, namely for any $z\in I$ we
have:
\[
|T_{j}^{-1}(z)-T_{j,\eta}^{-1}(z)|\leq\upsilon(\varepsilon)\ ;\ |T^{\prime
}(T_{j}^{-1}(z))-T_{\eta}^{\prime}(T_{j,\eta}^{-1}(z))|\leq\upsilon
(\varepsilon),\ j=1,2\ ,
\]
where $T_{j}^{-1},T_{j,\eta}^{-1}$ denote the inverse branches of the two
maps, $\upsilon(\varepsilon)\rightarrow0$ as $\varepsilon\downarrow0$ and in
the comparison of the derivatives we exclude $z=1.$
\end{itemize}

We now add two more assumptions \cite{BR}

\begin{description}
\item[A7] \textbf{Vertical closeness of the derivatives} For any $\eta\in
spt\lambda_{\varepsilon}$ let $k_{\eta}:=\inf\left\{  k\in\mathbb{N}%
:x_{0,\eta}\in B_{k\eta}\left(  x_{0}\right)  \right\}  $ be the the smallest
integer k for $k\eta$ be the radius of a ball centered in $x_{0}$ containing
the critical point of $T_{\eta}.$ We then assume that there exists a positive
constant $C$ such that
\[
\sup_{\eta\in spt\lambda_{\varepsilon}}\sup_{x\in B_{k_{\eta}\eta}^{c}(x_{0}%
)}\{|T_{\eta}^{\prime}(x)-T^{\prime}(x)|\}\leq C\upsilon(\varepsilon)\ .
\]

\item[A8] \textbf{Translational similarity of the branches} We suppose that,
for any $\eta\in spt\lambda_{\varepsilon},$ the branches $T_{i}%
:=T\upharpoonleft_{I_{i}}$ and $T_{i,\eta}:=T_{\eta}\upharpoonleft_{I_{i,\eta
}}$ corresponding to the same value of the index $i=1,2$ will not intersect
each other, but in $x=0,1.$
\end{description}

\begin{theorem}
\label{SSST}For any realization of the noise $\eta\in spt\lambda_{\varepsilon
},$ let $T_{\eta}$ satisfy the assumptions A1-A8. Then, $\mu_{T}$ is strongly
stochastically stable.
\end{theorem}

\begin{proof}
We use as $\mathbb{B}$ the Banach space of quasi-H\"{o}lder functions. Namely,
for all functions $h\in L^{1}$ and $0<\alpha\leq1$ we consider the seminorm
\[
|h|_{\alpha}:=\sup_{0<\varepsilon_{1}\leq\varepsilon_{0}}\frac{1}%
{\varepsilon_{1}^{\alpha}}\int\text{osc}(h,B_{\varepsilon_{1}}(x))dx\ ,
\]
where, for any measurable set $A,\text{ osc}(h,A):=\text{Essup}_{x\in
A}h(x)-\text{Essinf}_{x\in A}h(x).$ We say that $h$ belong to the set
$V_{\alpha}\subseteq L^{1}$ if $|h|_{\alpha}<\infty.V_{\alpha}$ does not
depend on $\varepsilon_{0}$ and equipped with the norm
\[
\left\Vert h\right\Vert _{\alpha}:=\left\vert h\right\vert _{\alpha
}+\left\Vert h\right\Vert _{1}%
\]
is a Banach space and from now on $V_{\alpha}$ will denote the Banach space
$\mathbb{B}:=(V_{\alpha},\left\Vert \cdot\right\Vert _{\alpha}).$ Furthermore,
it can be proved \cite{Sa} that $\mathbb{B}$ is continuously injected into
$L^{\infty}$ and in particular $||h||_{\infty}\leq C_{s}||h||_{\alpha}$ where
$C_{s}=\frac{\max(1,\varepsilon_{0}^{\alpha})}{\varepsilon_{0}^{n}}.$ Then, we
prove that the transfer operator for $T$ and for $T_{\eta}$ are close in the
norm $|||\cdot|||$ uniformly in $\eta$ which implies $||\left(  \mathcal{L}%
-\mathcal{L}_{\varepsilon}\right)  h||_{1}\leq O(\varepsilon)||h||_{\alpha}$
(see \cite{GV} Theorem 20).
\end{proof}

The proof of the result just sketched refers to the case where $T$ and its
perturbations are of the the Lorenz cusp-type map given in figs. 1 and 2.

The same technique can be used to show the stochastic stability of the
classical Lorenz-type map (see e.g. \cite{BR} fig. 1) again under the
uniformly expandingness assumption. In this case we do not need the vertical
closeness of the derivatives; instead we have to add the additional hypothesis
that the largest elongations between $|T(0)-T_{\eta}(0)|$ and $|T(1)-T_{\eta
}(1)|$ are of order $\varepsilon$ for any $\eta$ and moreover $|T_{1}%
^{-1}(T_{\eta}(0))|$ and $1-|T_{2}^{-1}(T_{\eta}(1))|$ are also of order
$\varepsilon,$ where the last two quantities are the size of the intervals
whose images contains points that have only one preimage when we apply
simultaneously the maps $T$ and $T_{\eta}.$ Hence they must be removed when we
compare the associate transfer operators. The proof then follows the same
lines of the previous one.


\begin{thebibliography}{99999}
\bibitem[Al]{Al}Alsmeyer G. \emph{The Markov Renewal Theorem and Related
Results} Markov Proc. Rel. Fields \textbf{3}, 103--127 (1997).

\bibitem[Ar]{Ar}Arnold L. \emph{Random Dynamical Systems }Springer (2003).

\bibitem[As]{As}Asmussen S., \emph{Applied Probability and Queues, II edition
}Springer (2003).

\bibitem[ABS]{ABS}V.S. Afraimovic, V.V. Bykov, Sili'nikov L.P. \emph{The
origin and structure of the Lorenz attractor} Dokl. Akad. Nauk SSSR
\textbf{234}, no. 2, 336--339 (1977).

\bibitem[AM]{AM}Ara\'{u}jo V., Melbourne I. \emph{Existence and smoothness of
the stable foliation for sectional hyperbolic attractors} Bull. London Math.
Soc. \textbf{49} 351--367 (2017).

\bibitem[AP]{AP}Ara\'{u}jo V., Pacifico M. J. \emph{Three-dimesional flows
}Springer (2010).

\bibitem[AS]{AS}Alves, J. F.,Soufi, M. \emph{Statistical stability of
geometric Lorenz attractors} Fundamenta Mathematicae \textbf{224}, 219--231 (2014).

\bibitem[BHV]{BHV}Bahsoun W., Hu H.-Y. Vaienti S. \emph{Pseudo-orbits,
stationary measures and metastability }Dyn. Syst. \textbf{29} n. 3 322--336 (2014).

\bibitem[BR]{BR}Bahsoun W., Ruziboev M. \emph{On the stability of statistical
properties for the Lorenz attractors with }$C^{1+\alpha}$ \emph{stable
foliation} Ergodic Theor. and Dyn. Sys. \textbf{39}, n.12, 3169--3184 (2019).

\bibitem[Bu]{Bu}Butterley O. \emph{Area expanding} $C^{1+\alpha}$
\emph{Suspension Semiflows} Commun. Math. Phys. \textbf{325} n.2, 803--820 (2014).

\bibitem[CHMV]{CHMV}G-P. Cristadoro, N. Haydn; Ph. Marie, S. Vaienti,
\emph{Statistical properties of intermittent maps with unbounded derivative}
Nonlinearity, \textbf{23} 1071-1095 (2010).

\bibitem[CMP]{CMP}S. Corti, F. Molteni, T. N. Palmer \emph{Signature of recent
climate change in frequencies of natural atmospheric circulation regimes}
Letters to Nature \textbf{398}, 799--802 (1999).

\bibitem[CSG]{CSG}Chekroun, M. D., Simonnet E., Ghil M. \emph{Stochastic
climate dynamics: random attractors and time-independent invariant measures}
Phisica D \textbf{240} n.21, 1685--1700 (2011).

\bibitem[Da]{Da}M. H. A. Davis \emph{Markov Models and Optimization} Springer (1993).

\bibitem[GL]{GL}Galatolo S., Lucena R. \emph{Spectral gap and quantitative
statistical stability for systems with contracting fibers and Lorenz-like
maps} Discrete Contin. Dyn. Syst. \textbf{40}, n.3, 1309--1360 (2020).

\bibitem[GMPV]{GMPV}Gianfelice M., Maimone F., Pelino V., Vaienti S. \emph{On
the recurrence and robust properties of the Lorenz'63 model} Commun. Math.
Phys. \textbf{313}, 745--779 (2012).

\bibitem[GV]{GV}Gianfelice M. Vaienti S. \emph{Stochastic Stability of the
Classical Lorenz Flow Under Impulsive Type Forcing} Journal of Statistical
Physics \textbf{181} n. 1, 163--211 (2020).

\bibitem[GW]{GW}Gukenheimer J., Williams R.F. \emph{Structural stability of
Lorenz attractors} Inst. Hautes Etudes Sci. Publ. Math. \textbf{50}, 59--72 (1979).

\bibitem[H-LL]{H-LL}Hern\'{a}ndez-Lerma O., Lasserre J. B. \emph{Further
criteria for positive Harris recurrence of Markov chains} Procedings of the
American Mathematical Society \textbf{129}, n.5, 1521-1524, (2000).

\bibitem[Ke]{Ke}Keller H. \emph{Attractors and bifurcations of the stochastic
Lorenz system} Report 389, Institut f\"{u}r Dynamische Systeme,
Universit\"{a}t Bremen (1996).

\bibitem[Ki]{Ki}Kifer Y. \emph{Random Perturbations of Dynamical Systems}
Birkh\"{a}user (1988).

\bibitem[KS]{KS}Korolyuk V., Swishchuk A. \emph{Semi-Markov Random Evolutions
}Springer (1995).

\bibitem[Letal]{Letal}Lucarini V., Faranda D., Milhazes de Freitas J. M.,
Gomes Monteiro Moreira de Freitas A. C., Holland M., Kuna T., Todd M., Vaienti
S. \emph{Extremes and Recurrence in Dynamical Systems} John Wiley \& Sons (2016).

\bibitem[KL]{KL}Keller G., Liverani C. \emph{Stability of the spectrum for
transfer operators} Ann. Scuola Norm. Sup. Pisa Cl. Sci. (4) \textbf{28} n. 1,
141--152 (1999).

\bibitem[Lo]{Lo}Lorenz E. N. \emph{Deterministic Nonperiodic Flow} J. Atmos.
Sci., vol. 20, 130--141 (1963).

\bibitem[Me]{Me}Metzger R. J. \emph{Stochastic Stability for Contracting
Lorenz Maps and Flows} Comm. Math. Phys. \textbf{212}, 277--296 (2000).

\bibitem[MT]{MT}Meyn S. Tweedie R. L. \emph{Markov Chains and Stochastic
Stability, Second Edition} Cambridge University Press (2009).

\bibitem[NVKDF]{NVKDF}Nevo G., Vercauteren N., Kaiser A., Dubrulle B., Faranda
D. \emph{A statistical-mechanical approach to study the hydrodynamic stability
of stably stratified atmospheric boundary layer} Phys. Rev. Fluids \textbf{2},
084603 (2017).

\bibitem[Op]{Op}Oprisan A. \emph{An Invariance Principle for Additive
Functionals of Semi-Markov Processes} Analytical and Computational Methods in
Probability Theory \textbf{10684} 409--420 (2017).

\bibitem[Pa]{Pa}Palmer T. N. \emph{A Nonlinear Dynamical Perspective on
Climate Prediction} Journal of Climate \textbf{12} n.2, 575--591 (1999).

\bibitem[Pi1]{Pi1}Pianigiani G. \emph{First return map and invariant measures}
Israel Journal of Mathematics \textbf{35}, n. 1-2, 32--48, (1980).

\bibitem[Pi2]{Pi2}Pianigiani G. \emph{Existence of invariant measures for
piecewise continuous transformations} Annales Polonici Matematici \textbf{XL},
39--45, (1981).

\bibitem[PM]{PM}Pelino V., Maimone F. \emph{Energetics, skeletal dynamics, and
long term predictions on Kolmogorov-Lorenz systems} Physical Review E,
\textbf{76}, 046214 (2007).

\bibitem[PP]{PP}Pasini A., Pelino V. \emph{A unified view of Kolmogorov and
Lorenz systems} Phys. Lett. A \textbf{275}, 435--445 (2000).

\bibitem[Sa]{Sa}Saussol, B. \emph{Absolutely continuous invariant measures for
multidimensional expanding maps }Israel Journal of Mathematics \textbf{116}
223--248 (2000).

\bibitem[Sc]{Sc}Schmallfu\ss , B. \emph{The random attractor of the stochastic
Lorenz system} Z. angew. Math. Phys. \textbf{48} 951--975 (1997).

\bibitem[Su]{Su}Sura P. \emph{A general perspective of extreme events in
weather and climate} Atmospheric Research \textbf{101} 1--21 (2011).

\bibitem[Tu]{Tu}W. Tucker \emph{A rigorous ODE solver and Smale's 14th
problem} Foundations of Computational Mathematics, \textbf{2:1} 53--117 (2002).
\end{thebibliography}
\end{document}